\pdfoutput=1
\documentclass[reqno]{amsart}
\usepackage[utf8]{inputenc}
\usepackage{color}
\usepackage{prettyref}
\usepackage{mathtools}
\usepackage{amstext}
\usepackage{amsthm}
\usepackage{amssymb}
\usepackage{stmaryrd}
\usepackage{microtype}
\usepackage[unicode=true,
 bookmarks=false,
 breaklinks=true,pdfborder={0 0 0},pdfborderstyle={},backref=false,colorlinks=true]
 {hyperref}
\hypersetup{
 pdfstartview=FitV,linkcolor=black,citecolor=black}

\makeatletter
\numberwithin{equation}{section}
\numberwithin{figure}{section}
\theoremstyle{plain}
\newtheorem{thm}{\protect\theoremname}[section]
\theoremstyle{remark}
\newtheorem{rem}[thm]{\protect\remarkname}
\theoremstyle{plain}
\newtheorem{cor}[thm]{\protect\corollaryname}
\theoremstyle{plain}
\newtheorem{prop}[thm]{\protect\propositionname}
\theoremstyle{definition}
\newtheorem{defn}[thm]{\protect\definitionname}
\theoremstyle{definition}
\newtheorem{example}[thm]{\protect\examplename}
\theoremstyle{plain}
\newtheorem{lem}[thm]{\protect\lemmaname}

\usepackage{bbm}
\usepackage{pgfplots}
\usepgfplotslibrary{colormaps,fillbetween}
\pgfplotsset{compat=1.15}
\usepackage{mathrsfs}
\usetikzlibrary{arrows}

\usepackage{txfonts}
\usepackage{textcomp}
\usepackage{xcolor}
\usepackage{dsfont}
\usepackage{graphicx}
\usepackage{caption}
\usepackage{subcaption}
\usepackage{float}
\usepackage{tikz}
\usepackage[foot]{amsaddr}
\usepackage{enumitem}

\renewcommand{\d}{\mathrm{d}}
\renewcommand{\rho}{\varrho}
\renewcommand{\epsilon}{\varepsilon}

\newcommand{\er}{\mathfrak{e}}

\DeclareMathOperator{\supp}{supp}
\DeclareMathOperator{\card}{card}

\renewcommand{\tilde}{\widetilde}
\renewcommand{\emptyset}{\varnothing}
\renewcommand{\phi}{\varphi}
\def\N{\mathbb N}
\def\d{\;\mathrm d}
\def\R{\mathbb{R}}
\def\Z{\mathbb{Z}}

\def\Q{\mathcal{Q}}

\def\J{\mathfrak{J}}
\def\GL{\tau}

\AtBeginDocument{
\DeclareFieldFormat[article]{title}{{#1}}
\DeclareFieldFormat[unpublished]{title}{{#1}}
\DeclareFieldFormat[proceedings]{title}{{#1}}
\DeclareFieldFormat[incollection]{title}{{#1}}
\DeclareFieldFormat[unpublished]{title}{{\em #1}}
\renewbibmacro{in:}{}
\DeclareFieldFormat[article]{pages}{#1}
}

\newrefformat{prop}{Proposition~\ref{#1}}
\newrefformat{fact}{Fact~\ref{#1}}
\newrefformat{exa}{Example~\ref{#1}}
\newrefformat{cor}{Corollary~\ref{#1}}
\newrefformat{assu}{Assumption~\ref{#1}}
\newrefformat{subsec}{Section~\ref{#1}}
\newrefformat{sec}{Section~\ref{#1}}
\newrefformat{def}{Definition~\ref{#1}}
\newrefformat{enu}{(\ref{#1})}

\makeatother

\usepackage[style=authoryear,backend=bibtex,style=alphabetic,  backref=false, giveninits=true,citestyle=alphabetic, natbib=false, url=false, isbn=false,doi=true,eprint=false,maxbibnames=99]{biblatex}
\providecommand{\corollaryname}{Corollary}
\providecommand{\definitionname}{Definition}
\providecommand{\examplename}{Example}
\providecommand{\lemmaname}{Lemma}
\providecommand{\propositionname}{Proposition}
\providecommand{\remarkname}{Remark}
\providecommand{\theoremname}{Theorem}

\begin{document}
\title[Quantization dimensions via Rényi dimensions]{Quantization dimensions of compactly supported probability measures
via Rényi dimensions}
\author{Marc Kesseböhmer}
\email{mhk@uni-bremen.de}
\author{Aljoscha Niemann}
\email{niemann1@uni-bremen.de}
\address{FB 3 – Mathematik und Informatik, University of Bremen, Bibliothekstr.
1, 28359 Bremen, Germany}
\author{Sanguo Zhu}
\email{sgzhu@jsut.edu.cn}
\address{School of Mathematics and Physics, Jiangsu University of Technology,
Changzhou 213001, China}
\begin{abstract}
We provide a complete picture of the upper quantization dimension 
 in terms of the Rényi dimension by proving that the upper quantization
dimension $\overline{D}_{r}(\nu)$ of order $r>0$ for an arbitrary compactly supported Borel
probability measure $\nu$ is given by its Rényi dimension at the
point $q_{r}$ where the $L^{q}$-spectrum of $\nu$ and the line
through the origin with slope $r$ intersect. In particular, this
proves the continuity of $r\mapsto\overline{D}_{r}(\text{\ensuremath{\nu)}}$
as conjectured by Lindsay (2001). This viewpoint also sheds new light
on the connection of the quantization problem with other concepts
from fractal geometry in that we obtain a one-to-one correspondence
of the upper quantization dimension and the $L^{q}$-spectrum restricted
to $\left(0,1\right)$. We give sufficient conditions in terms of
the $L^{q}$-spectrum for the existence of the quantization dimension.
In this way we show as a byproduct that the quantization dimension
exists for every Gibbs measure with respect to a $\mathcal{C}^{1}$-self-conformal
iterated function system on $\R^{d}$ without any assumption on the
separation conditions as well as for inhomogeneous self-similar measures
under the inhomogeneous open sets condition. Some known general bounds
on the quantization dimension in terms of other fractal dimensions
can readily be derived from our new approach, some can be improved.
\end{abstract}

\maketitle

\section{Introduction and statement of main results\label{sec:Introduction-and-background}}

The quantization problem for probability measures originates from
information theory, e.\,g\@. image compression and data processing,
and was subsequently studied in great detail by many mathematicians.
Recently, this theory has again attracted increasing attention in
applications such as optimal transport problems \cite{MR4359783},
numerical integration \cite{pages2015introduction,MR4337711} and
mathematical finance \cite{MR2738823,MR3287772,MR3451833,MR3455415,MR4029233}.
The core problem is to study the asymptotics of the errors in approximating
a given random variable with a quantized version of the random variable
(i.\,e\@. taking only finitely many values), in the sense of $r$-means.
The quantization dimension then reflects the exponential rate of this
convergence and it has been studied by various authors, for example
\cite{MR2182671,MR1877974,MR2158264}, \cite{MR3035159,MR3300877,MR3345508,MR3558153,MR3415729,MR3489058,MR3632842,MR3636381,MR3731698,MR4119544,MR4206065}.
To be more specific, let $X$ be a bounded $\R^{d}$-valued random
variable, $d\in\N$, on a probability space $\left(\Omega,\mathcal{A},\mathbb{P}\right)$
with distribution $\nu\coloneqq\mathbb{P}\circ X^{-1}$. For a given
$n\in\N$, let $\mathcal{F}_{n}$ denote the set of all Borel measurable
functions $f:\R^{d}\rightarrow\R^{d}$ with $\card\left(f\left(\R^{d}\right)\right)\leq n$;
we call an element of $\mathcal{F}_{n}$ an $n$-quantizer. Our aim
is to approximate $X$ with a quantized version of $X$, i.\,e\@.
$X$ will be approximated by $f\circ X$ with $f\in\mathcal{F}_{n}$
where we quantify this approximation with respect to the $r$-quasinorm.
More precisely, we are interested in the\emph{ $n$-th quantization
error} of $\nu$\emph{ of order $r>0$} given by
\[
\er_{n,r}(\nu)\coloneqq\inf_{f\in\mathcal{F}_{n}}\left(\int_{\Omega}\left|X-f(X)\right|^{r}\d\mathbb{P}\right)^{1/r}=\inf_{f\in\mathcal{F}_{n}}\left(\int\left|x-f(x)\right|^{r}\d\nu(x)\right)^{1/r}.
\]
As pointed out for instance in \cite{MR1877974}, `\emph{the problem
of determining the quantization dimension function for a general probability
is open.}` With the present paper we want to close this gap completely
for the upper quantization dimension and, under additional regularity
conditions, also for the lower quantization dimension. In fact, we
identify the upper quantization dimension of $\nu$ of order $r>0$
as the Rényi dimension of $\nu$ evaluated at the point $q_{r}$ where
the $L^{q}$-spectrum of $\nu$ and the line through the origin with
slope $r$ intersect. Building on a result of \parencite{MR1838304,MR2322179},
we prove the existence of the quantization dimension of Gibbs measures
with respect to conformal iterated function systems without any separation
conditions and, by a result of \cite{MR3201323}, we prove the same
statement for inhomogeneous self-similar measures under the iOSC (see
below for definitions).

The starting point of our investigations was the observation that
for particular measures the quantization dimension $D_{1}\left(\nu\right)$
of order $r=1$ is connected to the upper spectral dimension $\overline{s}_{\nu}$
of the Kre\u{\i}n–Feller operator associated to $\nu$ for $d=1$
as determined in \cite{KN2022} via the identity $\overline{D}_{1}\left(\nu\right)=\overline{s}_{\nu}/\left(1-\overline{s}_{\nu}\right)$.
Indeed, for $r\neq1$, we also expect similar connections to higher
dimensional polyharmonic operators as considered in \cite{KN21b}.

In the following we assume that $\nu$ is a compactly supported Borel
probability measure and let $\mathcal{D}_{n}$ denote the partition
of $\R^{d}$ by half-open cubes of the form $\prod_{i=1}^{d}\left(k_{i}2^{-n},\left(k_{i}+1\right)2^{-n}\right]$
with $k\in\Z^{d}$. We set $\mathcal{D}\coloneqq\bigcup_{n\in\N}\mathcal{D}_{n}$,
which defines a semiring of sets. For every $n\in\mathbb{N}$, we
write $\mathcal{A}_{n}:=\left\{ \alpha\subset\mathbb{R}^{d}:1\leq{\rm card}(\alpha)\leq n\right\} $.
Due to \cite[Lemma 3.1]{MR1764176} an equivalent formulation of the
$n$-th quantization error of $\nu$ of order $r$ is given by 
\begin{equation}
\er_{n,r}(\nu)=\left\{ \begin{array}{ll}
\inf_{\alpha\in\mathcal{A}_{n}}\big(\int d(x,\alpha)^{r}\d\nu(x)\big)^{1/r}, & r>0,\\
\inf_{\alpha\in\mathcal{A}_{n}}\exp\int\log d(x,\alpha)\d\nu(x), & r=0,
\end{array}\right.\label{eq:quanerrordef}
\end{equation}
with $d(x,\alpha)\coloneqq\min_{y\in\alpha}\left\Vert x-y\right\Vert $.
By \cite[Lemma 6.1]{MR1764176}, we have $\er_{n,r}(\nu)\rightarrow0$.
In fact, it is well known that $\er_{n,r}(\nu)=O(n^{-1/d})$ and,
if $\nu$ is singular with respect to Lebesgue, then $\er_{n,r}(\nu)=o\left(n^{-1/d}\right)$
(see \cite[Theorem 6.2]{MR1764176}). It is a natural question to
find the `right exponent' for the convergence order. This will be
the main concern of this paper. For this we define the \emph{upper
and lower quantization dimension for $\nu$ of order $r$} by 
\[
\overline{D}_{r}(\nu):=\limsup_{n\to\infty}\frac{\log n}{-\log\er_{n,r}(\nu)},\;\;\underline{D}_{r}(\nu):=\liminf_{n\to\infty}\frac{\log n}{-\log\er_{n,r}(\nu)}.
\]
If $\overline{D}_{r}(\nu)=\underline{D}_{r}(\nu)$, we call the common
value the \emph{quantization dimension for $\nu$ of order $r$} and
denote it by $D_{r}(\nu)$. Note that without loss of the generality,
we can (and for ease of exposition, we will) assume that the support
of $\text{\ensuremath{\nu}}$ is contained in $\Q\coloneqq(0,1]^{d}$.
To see this fix $a\in\R\setminus\left\{ 0\right\} $, $b\in\R^{d}$
and let $\Phi_{a,b}(x)\coloneqq ax+b$, $x\in\R^{d}$ such that $\Phi_{a,b}(\supp(\nu))\subset(0,1]^{d}$.
Then, 
\begin{align*}
\er_{n,r}\left(\nu\circ\Phi_{a,b}^{-1}\right) & =\inf_{f\in\mathcal{F}_{n}}\left(\int\left|\Phi_{a,b}(x)-f(\Phi_{a,b}(x))\right|^{r}\d\nu\right)^{1/r}\\
 & =\inf_{f\in\mathcal{F}_{n}}\left(\int\left|ax+b-f(\Phi_{a,b}(x))\right|^{r}\d\nu\right)^{1/r}\\
 & =|a|\inf_{f\in\mathcal{F}_{n}}\left(\int\left|x-\Phi_{1/a,-b/a}\left(f(\Phi_{a,b}(x))\right)\right|^{r}\d\nu\right)^{1/r}=|a|\er_{n,r}\left(\nu\right),
\end{align*}
where we have used the fact that $f\mapsto\Phi_{1/a,-b/a}\circ f\circ\Phi_{a,b}$
defines a surjection on $\mathcal{F}_{n}$. The following classical
result in quantization theory goes back to Zador \cite{MR651809}
and was generalized by Bucklew and Wise in \cite{MR651819}; we refer
to \cite[Theorem 6.2]{MR1764176} for a rigorous proof:
\begin{quote}
\emph{Let $\nu$ be a Borel probability measure with bounded support
and non-vanishing absolutely continuous part with density $h$. Then,
for }$r>0$, \emph{
\[
\lim_{n\rightarrow\infty}n^{-r/d}\er_{n,r}(\nu)^{r}=C(r,d)\bigg(\int h^{\frac{d}{d+r}}(x)\d x\bigg)^{\frac{d+r}{d}},
\]
where $C(r,d)$ is a constant independent of $\nu$.}
\end{quote}
Interestingly, there is an analog results on the eigenvalue counting
function for polyharmonic operators (see for instance \cite{MR0278126}).
While engineers are mainly dealing with absolutely continuous distributions,
the quantization problem is significant for all Borel probability
measures with bounded support.

In this note we follow some ideas that are developed in \cite{KN2022,KN2022b,KN21,KN21b}
to tackle spectral asymptotics for Kre\u{\i}n–Feller operators. One
of the central objects is the $L^{q}$-spectrum $\beta_{\nu}$ of
$\nu$ given, for $q\in\R_{\geq0}$, by
\[
\beta_{\nu}(q)\coloneqq\limsup_{n\rightarrow\infty}\beta_{\nu,n}(q)\;\;\;\:\text{ with }\:\;\;\;\beta_{\nu,n}(q)\coloneqq\frac{\log\left(\sum_{C\in\mathcal{D}_{n}}\nu(C)^{q}\right)}{\log(2^{n})}.
\]
 Define
\[
q_{r}\coloneqq\inf\left\{ q>0:\beta_{\nu}(q)<rq\right\} .
\]
Further, the \emph{(upper) generalized Rényi dimension of $\nu$ }\cite{zbMATH03124309}
is given, for $q\in\R_{\geq0}$, by
\[
\overline{\mathfrak{R}}_{\nu}\left(q\right)\coloneqq\begin{cases}
\beta_{\nu}(q)/\left(1-q\right), & \text{for }q\neq1,\\
\limsup_{n}\left(\sum_{C\in\mathcal{D}_{n}}\nu(C)\log\nu(C)\right)/\log\left(2^{-n}\right), & \text{for }q=1.
\end{cases}
\]
It turns out that, for $q\neq1$ (and similar for $q=1$), the Rényi
dimension can be expressed also in terms of the Hentschel–Procaccia
generalized dimension \cite{MR719636}
\[
\overline{\mathfrak{R}}_{\nu}\left(q\right)=\frac{1}{1-q}\limsup_{r\to0}\frac{\log\int\nu\left(B_{r}\left(x\right)\right)^{q-1}\d\nu\left(x\right)}{-\log r}.
\]
It is easy to construct purely atomic measures such that $q_{r}=0$
for all $r>0$ and $\beta_{\nu}(0)>0$, see \cite{KN2022}. In this
case, the upper quantization dimension is also $0$.
Further, we will need some ideas from entropy theory: Let us define
the set of $\nu$-\emph{partitions} $\Pi_{\nu}$ to be the set of
finite collections $P$ of dyadic cubes such that there exists a partition
$\tilde{P}$ of $\Q$ by dyadic cubes from $\mathcal{D}$ with $P=\left\{ Q\in\tilde{P}:\nu(Q)>0\right\} $.
We define
\begin{equation}
\mathcal{M}_{\nu,r}\left(x\right)\coloneqq\inf\left\{ \card\left(P\right):P\in\Pi_{\nu},\,\max_{Q\in P}\nu\left(Q\right)\Lambda\left(Q\right)^{r/d}<1/x\right\} .\label{eq:M-nu(r)}
\end{equation}
and
\[
\overline{h}_{\nu,r}\coloneqq\limsup_{x\to\infty}\frac{\log\mathcal{M}_{\nu,r}\left(x\right)}{\log x},\quad\underline{h}_{\nu,r}\coloneqq\liminf_{x\to\infty}\frac{\log\mathcal{M}_{\nu,r}\left(x\right)}{\log x}
\]
will be called the\emph{ upper, }resp.\emph{ lower, $\left(\nu,r\right)$-partition
entropy.} We write $\overline{h}_{r}\coloneqq\overline{h}_{\nu,r}$
and $\underline{h}_{r}\coloneqq\underline{h}_{\nu,r}$ .

For all $n\in\N$ and $\alpha>0$, we define 
\[
\mathcal{N}_{\nu,\alpha,r}\left(n\right)\coloneqq\card N_{\nu,\alpha,r}\left(n\right),\quad N_{\nu,\alpha,r}\left(n\right)\coloneqq\left\{ Q\in\mathcal{D}_{n}:\nu\left(Q\right)\Lambda\left(Q\right)^{r/d}\geq2^{-\alpha n}\right\} ,
\]
and set
\[
\overline{F}_{\nu,r}\left(\alpha\right)\coloneqq\limsup_{n}\frac{\log^{+}\left(\mathcal{N}_{\nu,\alpha,r}\left(n\right)\right)}{\log2^{n}}\text{ and }\underline{F}_{{\nu,r}}\left(\alpha\right)\coloneqq\liminf_{n}\frac{\log^{+}\left(\mathcal{N}_{\nu,\alpha,r}\left(n\right)\right)}{\log2^{n}}.
\]
Following \cite{KN2022b}, we refer to the quantities
\[
\overline{F}_{r}\coloneqq\overline{F}_{\nu,r}\coloneqq\sup_{\alpha>0}\frac{\overline{F}_{\nu,r}\left(\alpha\right)}{\alpha}\text{ and }\underline{F}_{r}\coloneqq\underline{F}_{\nu,r}\coloneqq\sup_{\alpha>0}\frac{\underline{F}_{\nu,r}\left(\alpha\right)}{\alpha}
\]
as the $\left(\nu,r\right)$\emph{-upper}, resp. \emph{lower, optimized
coarse multifractal dimension}. We will see in \prettyref{sec:OptimalPartitions}
that we always have 
\begin{equation}
\overline{F}_{r}=q_{r}=\overline{h}_{r}\geq\underline{h}_{r}\geq\underline{F}_{r}.\label{eq:identities}
\end{equation}
Now, we are in the position to state our main results. All proofs
of this section will be postponed to \prettyref{subsec:Proofs-of-main}.
\begin{thm}
\label{thm:Main}Let $\nu$ be compactly supported probability measure
on $\R^{d}$. If $\sup_{x\in(0,1)}\beta_{\nu}(x)>0$, then for every
$r>0$, 
\[
\frac{r\underline{F}_{r}}{1-\underline{F}_{r}}\leq\underline{D}_{r}\left(\nu\right)\leq\frac{r\underline{h}_{r}}{1-\underline{h}_{r}}\leq\overline{D}_{r}\left(\nu\right)=\overline{\mathfrak{R}}_{\nu}\left(q_{r}\right).
\]
Otherwise, if $\beta_{\nu}(x)=0$ for all $x>0$, then $D_{r}\left(\nu\right)=0$
for all $r\geq0$.
\end{thm}

\begin{rem}
We would like to emphasize here that the proof of the upper bound
of $\overline{D}_{r}\left(\nu\right)$ is based on an \emph{adaptive
approximation algorithm} in the sense of \cite{MR939183,MR1781213}.
This can be used to implement a straightforward procedure for finding
a quantizer $f\in\mathcal{F}_{n},$ $n\in\N$ that provides an upper
bound, which is optimal with respect to the upper exponential rate.
\end{rem}

\begin{rem}
~
\begin{enumerate}
\item The following useful identities are a consequence of \prettyref{eq:identities}
and apply independently of the condition on $\beta_{\nu}$:
\[
\overline{D}_{r}\left(\nu\right)=\frac{r\overline{h}_{r}}{1-\overline{h}_{r}}=\frac{rq_{r}}{1-q_{r}}=\frac{r\overline{F}_{r}}{1-\overline{F}_{r}}.
\]
We will make use of this observation in the proof of \prettyref{thm:Main}.
\item At least for special cases, it has been observed in \cite[p. 6]{MR3558153}
that the upper quantization dimension is often determined by a critical
value; we are now in the position to determine this critical value
for arbitrary compactly supported probability measures as 
\[
\overline{D}_{r}\left(\nu\right)=\inf\left\{ q>0:\sum_{Q\in\mathcal{D}}\left(\nu(Q)\Lambda(Q)^{r/d}\right)^{q/\left(q+r\right)}<\infty\right\} .
\]
This follows from the definition of $q_{r}$ and the identities given
in (1) (see also \cite[Lemma 3.3]{KN2022b}).
\item For $r=0$ a general result of \cite{MR2954531} is applicable yielding
\[
\dim_{H}\left(\nu\right)\leq\underline{D}_{0}\left(\nu\right)\leq\overline{D}_{0}\left(\nu\right)\leq\dim_{P}^{*}\left(\nu\right)\leq-\partial^{-}\beta(1),
\]
where $-\partial^{-}\beta(1)$ denotes the left-sided derivative of
$\beta$ in $1$, 
\begin{align*}
\dim^{*}\left(\nu\right) & \coloneqq\inf\left\{ \dim\left(A\right):\mu\left(A\right)=\mu\left(\R^{d}\right)\right\} ,\\
\dim\left(\nu\right) & \coloneqq\inf\left\{ \dim\left(A\right):\mu\left(A\right)>0\right\} 
\end{align*}
with $\dim$ denoting either the Hausdorff dimension $\dim_{H}$ or
the packing dimension $\dim_{P}$.
\item By \cite{MR1866392}, for $r=\infty$ and defining the quantization
dimension $\overline{D}_{\infty}(\nu)$ via the quantization error
of $\nu$ of order $\infty$ in terms of $L_{\infty}\left(\nu\right)$
in the place of $L_{r}\left(\nu\right)$ in \prettyref{eq:quanerrordef},
we have
\[
\lim_{r\uparrow\infty}\overline{D}_{r}\left(\nu\right)=\lim_{x\downarrow0}\beta_{\nu}(x)\leq\overline{D}_{\infty}(\nu)=\overline{\dim}_{M}\left(\nu\right)\coloneqq\overline{\dim}_{M}(\supp(\nu))
\]
and 
\begin{equation}
\underline{D}_{r}\left(\nu\right)\leq\underline{\dim}_{M}\left(\nu\right)\coloneqq\underline{\dim}_{M}(\supp(\nu)),\label{eq:PoetzelbergerlowerMinkowski}
\end{equation}
where $\overline{\dim}_{M}\left(A\right)$ denotes the upper and $\underline{\dim}_{M}\left(A\right)$  the lower Minkowski
dimension of $A\subset\mathbb{R}^{d}$.  
\end{enumerate}
\end{rem}

\begin{rem}
By \prettyref{thm:Main} we infer the following one-to-one correspondence
between $\overline{D}_{r}\left(\nu\right)$, $r>0$ and $\beta_{\nu}\left(q\right)$,
$q\in(0,1$). For this note that $q_{r}=\overline{D}_{r}\left(\nu\right)/\left(r+\overline{D}_{r}\left(\nu\right)\right)$.
Hence, if $\overline{D}_{r}\left(\nu\right)>0$, then
\[
\beta_{\nu}\left(\frac{\overline{D}_{r}\left(\nu\right)}{\overline{D}_{r}\left(\nu\right)+r}\right)=\frac{r\overline{D}_{r}\left(\nu\right)}{\overline{D}_{r}\left(\nu\right)+r}.
\]
If we now set $q_{0}\coloneqq\sup\left\{ \overline{D}_{r}\left(\nu\right)/\left(r+\overline{D}_{r}\left(\nu\right)\right):r>0\right\} $,
then for all $q\in(q_{0},1]$ we have $\beta_{\nu}(q)=0$ and for
$q\in\left(0,q_{0}\right)$, we have $\beta_{\nu}(q)>0$ and
\[
\overline{D}_{\beta_{\nu}(q)/q}\left(\nu\right)=\frac{\beta_{\nu}\left(q\right)}{1-q}.
\]
At the end of this section in \prettyref{exa:ZhuiSelfSimilar}, we
exploit the fact that the upper quantization dimension is known for
inhomogeneous self-similar measures supported on self-similar sets
(see \cite{MR3454048}) and derive from this its $L^{q}$-spectrum
on $(0,1)$.
\end{rem}

Further, using $\overline{D}_{r}\left(\nu\right)=rq_{r}/\left(1-q_{r}\right),$
we can affirm a conjecture of Lindsay stated in his PhD thesis \cite{MR2704080}:
\begin{cor}[Continuity]
\label{cor:DrContinuous}The map $r\mapsto\overline{D}_{r}(\nu)$
is continuous on $\left(0,\infty\right)$.
\end{cor}

As another easy consequence, we reproduce a previous result of \cite{MR2326060}
on the stability of quantization dimension.
\begin{cor}[Finite stability]
\label{cor:mixed measure}For any two compactly supported probability
measures $\nu_{1}$ and $\nu_{2}$ , $p\in(0,1)$ and each $r>0$,
we have 
\[
\overline{D}_{r}\left(p\nu_{1}+(1-p)\nu_{2}\right)=\max\left\{ \overline{D}_{r}\left(\nu_{1}\right),\overline{D}_{r}\left(\nu_{2}\right)\right\} .
\]
\end{cor}

Our approach gives also rise to a slight improvement for the upper
bound of the lower quantization dimension. For this we introduce the
following quantity
\[
\dim_{\infty}\left(\nu\right)\coloneqq\liminf_{n\to\infty}\frac{\log\max_{Q\in\mathcal{D}_{n}}\nu\left(Q\right)}{-\log2^{n}}
\]
to sharpen \prettyref{eq:PoetzelbergerlowerMinkowski} from \cite{MR1866392}.
\begin{prop}
\label{prop:upperBoundForLowerDr} For $r>0$ such that $\underline{\dim}_{M}(\nu)/\left(r+\dim_{\infty}(\nu)\right)<1$,
we have
\[
\underline{D}_{r}\left(\nu\right)\leq\frac{r\underline{\dim}_{M}(\nu)}{r+\dim_{\infty}(\nu)-\underline{\dim}_{M}(\nu)}.
\]
\end{prop}

\begin{rem}
If the quantization dimension exists for some $r>0$, then it is given
by purely measure-geometric data encoded by the $\left(\nu,r\right)$-partition
entropy, namely we have $\overline{h}_{r}=\underline{h}_{r}=h_{r}$
and this value determines the quantization dimension $D_{r}\left(\nu\right)=rh_{r}/\left(1-h_{r}\right)$.

As a further consequence of our formalism, we derive a rigidity statement
that forces the quantization dimension to be constant in $r>0$.
\end{rem}

\begin{cor}
\label{cor:rigidity}Assume $\overline{\dim}_{M}(\nu)>0$. Then the
following statements are equivalent:
\begin{enumerate}
\item $\overline{D}_{r}\left(\nu\right)=\overline{\dim}_{M}(\nu)$ for some
$r>0$
\item $\overline{\mathfrak{R}}_{\nu}\left(q\right)=\overline{\dim}_{M}(\nu)$
for some $q\in\left(0,1\right)$.
\item $\partial^{-}\beta_{\nu}\left(1\right)\leq-\overline{\dim}_{M}(\nu)$.
\item The $L^{q}$-spectrum is linear on $[0,1]$, i.\,e\@. $\beta_{\nu}(q)=\overline{\dim}_{M}(\nu)(1-q)$,
$q\in[0,1]$.
\item $\overline{D}_{r}\left(\nu\right)=\overline{\dim}_{M}(\nu)$ for all
$r>0$.
\end{enumerate}
In particular, we find the necessary condition $\dim_{P}^{*}\left(\mu\right)=\overline{\dim}_{M}(\nu)$
or $\dim_{H}\left(\mu\right)=\overline{\dim}_{M}(\nu)$, which would
imply (3) of the above statements.
\end{cor}

As a second main result, we find often easily verifiable necessary
conditions that guarantee upper and lower quantization dimension to
match.
\begin{defn}
We define two notions of regularity.
\begin{enumerate}
\item We call $\nu$ \emph{multifractal-regular in $r$ ($r$-MF-regular),}
if $\underline{F}_{\nu,r}=\overline{F}_{\nu,r}$.
\item We call the measure $\nu$ \emph{$L^{q}$-regular for $r$} if
\begin{enumerate}
\item $\beta_{\nu}\left(q\right)=\liminf_{n}\beta_{\nu,n}\left(q\right)\in\R$
for all $q\in\left(q_{r}-\varepsilon,q_{r}\right)$, for some $\varepsilon>0$,
or\label{enu:-LqRegularity1}
\item $\beta_{\nu}\left(q_{r}\right)=\liminf_{n}\beta_{\nu,n}\left(q_{r}\right)$
and $\beta_{\nu}$ is differentiable in $q_{r}$.
\end{enumerate}
\end{enumerate}
\end{defn}

The following theorem shows that the spectral partition function is
a valuable auxiliary concept to determine the quantization dimension
for a given measure $\nu$.
\begin{thm}
\label{thm:LqRegularImpliesRegular} The following regularity implications
hold for $r>0$:
\[
\nu\,\,\text{is }L^{q}\text{-regular for }r\ensuremath{\implies}\nu\,\,\text{is \ensuremath{r}-MF-regular \ensuremath{\implies}}\underline{D}_{r}\left(\nu\right)=\overline{D}_{r}\left(\nu\right)=\overline{\mathfrak{R}}_{\nu}\left(q_{r}\right).
\]
\end{thm}

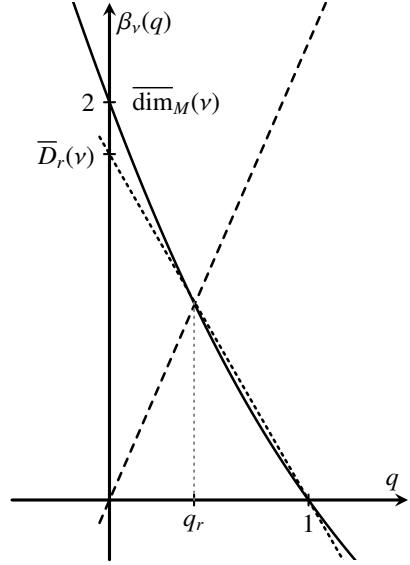
\begin{figure}
\center{\begin{tikzpicture}[scale=0.98, every node/.style={transform shape},line cap=round,line join=round,>=triangle 45,x=1cm,y=1cm] \begin{axis}[ x=2.7cm,y=2.7cm, axis lines=middle, axis line style={very thick},ymajorgrids=false, xmajorgrids=false, grid style={thick,densely dotted,black!20}, xlabel= {$q$}, ylabel= {$\beta_\nu (q)$}, xmin=-0.49 , xmax=1.5 , ymin=-0.3, ymax=2.5,x tick style={thick, color=black}, xtick={0, .425,1},xticklabels = {0,$q_r$,1}, y tick style={thick, color=black}, ytick={0,1/0.575,2},yticklabels = {0,$\overline{D}_{r}(\nu)$,2}] \clip(-0.5,-0.3) rectangle (4,4); 
\draw[line width=1pt,smooth,samples=180,domain=-0.3:3.4] plot(\x,{log10(0.08^((\x))+0.2^((\x))+0.06^((\x))+0.66^((\x)))/log10(2)}); 
%\draw[line width=2pt, smooth,samples=100,domain=-0.2:4.4] plot(\x,{(\x)}); 
%\draw [line width=1pt,dashed, domain=-0.05 :1.4] plot(\x,{0.71*(0.57-\x)+0.2}); 
%\draw [line width=01pt,dotted, domain=-0.05 :4.4] plot(\x,{(((log10(0.36))/(log10(2))+1)*(\x)+2)});
\draw [line width=01pt,dotted, domain=-0.05 :4.4] plot(\x,{1/(1-0.425)*(1-\x)});
\draw [line width=01pt,dashed, domain=-0.05 :4.4] plot(\x,{2.3*\x});
 
%\fill (0.57 ,0.2) circle[radius=2pt];
%\fill (2.48 ,0 ) circle[radius=2pt];
\node[circle,draw] (c) at (2.48 ,0 ){\,};
%\draw [line width=1pt ] (0.03 ,0.2)–(-0.03,0.2);
%\draw [line width=1pt ] (0.03 ,0.355)–(-0.03,0.355);
%\draw [line width=1pt ] (0.03 ,1)–(-0.03,1);
%\draw [line width=1pt ] (0.03 ,.38)–(-0.03,0.38);
%\draw [line width=1pt ] (0.03 ,.288)–(-0.03,.288);
%\draw [line width=1pt ] (0.03 ,0.605)–(-0.03,0.605 );

\draw [line width=.7pt,dotted, gray] (0.425 ,0.)--(0.425,1); 
%\draw [line width=.7pt,dotted, gray] (0  ,1.0 )-- (1,1);
%\draw [line width=.7pt,dotted, gray] (0.38  ,0.38 )–(0,0.38 );
%\draw [line width=.7pt,dotted, gray] (0.3550 ,0.3550)–(-0.15,0.355 );
%\draw [line width=.7pt,dotted, gray] (-0.15,0.355 )–(-0.18,0.43 );
%\draw [line width=.7pt, dotted , gray] (-0.03 ,0.440)–(-0.15,0.44 );
%\draw (2.34 ,-0.01) node[anchor=north west] {{$q^N$}}; 
%\draw (-0.31,2.48) node[anchor=north west] {$\GL_{\J_\nu}^N(q)$}; 
%\draw (-0.25 ,0.68 ) node[anchor=north west] {$\Large \widehat\beta_\nu(-\alpha)$}; 
%\draw (-0.13 ,0.43 ) node[anchor=north west] {$\Large \displaystyle{s_{\nu}}$};
%\draw (3.37 ,-0.02) node[anchor=north west]{$q$};
%\draw (-0.31 ,0.345 ) node[anchor=north west] {$\Large \displaystyle{\dim_H(\nu) }$};
\draw (0.07 ,2.13 ) node[anchor=north west] {$ \displaystyle{\overline{\dim}_M(\nu)}$}; 
\end{axis} 
\end{tikzpicture}}

\caption{\label{fig:Moment-generating-function}The $L^{q}$-spectrum $\beta_{\nu}$
for the self-similar measure $\nu$ supported on the\emph{ Sierpi\'{n}ski
tetraeder} in $\R^{3}$ with all four contraction ratios equal $1/2$
and with probability vector $\left(0.66,0.2,0.08,0.06\right)$. $\beta_{\nu}\left(0\right)=\overline{\dim}_{M}\left(\nu\right)=2$.
For $r=2.3$ (slope of the dashed line) the intersection of the graph
of $\beta_{\nu}$ and the dashed line determines $q_{r}$. The dotted
line through $\left(q_{r},\beta_{\nu}\left(q_{r}\right)\right)$ and
$\left(0,1\right)$ intersects the vertical axis in $\overline{D}_{r}\left(\nu\right)$.}
\end{figure}

Special cases for the following corollary have a long history where
the results concerning the quantization dimension were obtained only
under the open set condition, strong open set condition or even the
strong separation condition \cite{MR1434978,MR1844394,MR1912380,MR2776065,MR3632842,MR3636381}.
Unlike here, these papers investigated also the quantization coefficients,
which provide more accurate information about the approximation rate.
Recall e.\,g\@. Graf and Luschgy's results \cite{MR1844394,MR2055056}
for self-similar measures: Let $(f_{i})_{i=1}^{N}$ be a family of
contractive similitudes on $\mathbb{R}^{d}$ with contraction ratios
$\left(r_{i}\right)_{i=1}^{N}\in\left(0,1\right)^{N}$. According
to \cite{MR625600}, for a given probability vector $(p_{i})_{i=1}^{N}$,
there exists a unique compactly supported Borel probability measure
on $\mathbb{R}^{d}$ such that $\nu=\sum_{i=1}^{N}p_{i}\cdot\nu\circ f_{i}^{-1}$.
This measure is called the \emph{self-similar measure.} We say that
$(f_{i})_{i=1}^{N}$ satisfies the \emph{open set condition} (OSC)
if there exists a non-empty open set $U$ such that $f_{i}(U)\cap f_{j}(U)=\emptyset$
for all $i\neq j$, and $f_{i}(U)\subset U$ for all $i=1,\cdots,N$.
If additionally $\supp\nu\cap U\neq\emptyset$ for the open set $U$
in the OSC, we say that the \emph{strong open set condition} (SOSC)
holds. For $r\in(0,\infty)$, let $k_{r}$ be the positive real number
given by 
\begin{equation}
\sum_{i=1}^{N}\left(p_{i}r_{i}^{r}\right)^{k_{r}/\left(k_{r}+r\right)}=1,\label{eq:gl-1}
\end{equation}
then, assuming the OSC, the quantization dimension is given by $D_{r}(\nu)=k_{r}.$
This result has been generalized by Lindsay (\cite{MR2704080}) to
conformal IFS under the strong open set condition (see also \cite{MR1877974})
using the notion of the associated pressure function. We are able
to generalize this result further as follows:

Let $U\subset\R^{d}$ be an open set. We say a $C^{1}$-map $S:U\rightarrow\R^{d}$
is \emph{conformal }if for every $x\in U$ the matrix $S'(x)$, giving
the total derivative of $S$ in $x$, satisfies $|S'(x)\cdot y|=\left\Vert S'(x)\right\Vert |y|$
for all $y\in\R^{d}$\emph{ and $\left\Vert S'_{u}(x)\right\Vert \coloneqq\sup_{|z|=1}|S'_{u}(x)\cdot z|>0$.}
A family of conformal mappings $\left(S_{i}:K\rightarrow K\right)_{i\in\left\{ 1,\ldots,N\right\} }$,
for some compact set $K$ and $N\geq2$ is a $C^{1}$-\emph{conformal
iterated function system }($C^{1}$-cIFS) if for each $i\in I$, the
contraction $S_{i}$ extends to an injective conformal map $S_{i}:U\rightarrow U$
on an open set $U\supset K$ and such that \emph{$\sup\left\{ \left\Vert S'_{i}(x)\right\Vert :x\in U\right\} <1$.
}The unique set $L$ such that $L=\bigcup_{i=1}^{N}S_{i}\left(L\right)$
is called the \emph{self-conformal set }with respect to $\left(S_{i}\right)_{i\in\left\{ 1,\ldots,N\right\} }$.
For any Gibbs measure $\mu$ on the full shift $\left\{ 1,\ldots,N\right\} ^{\N}$
we call $\nu\coloneqq\mu\circ\pi^{-1}$ the Gibbs measure with respect
to a $\left(S_{i}\right)_{i\in\left\{ 1,\ldots,N\right\} }$ whenever
$\pi:\left\{ 1,\ldots,N\right\} ^{\N}\to L$ denotes the canonical
coding map.
\begin{cor}
\label{cor:sefl-conformalWithOverlap} Let $\nu$ be a Gibbs measure
with respect to a $C^{1}$-cIFS on $\R^{d}$ with no assumptions on
the separation conditions. Then, for $r>0$, 
\[
D_{r}\left(\nu\right)=\frac{rq_{r}}{1-q_{r}}.
%\overline{\mathfrak{R}}_{\nu}\left(q_{r}\right)%
\]
\end{cor}

\begin{rem}
The study of the $L^{q}$-spectra is a highly active research area.
For instance, the $L^{q}$-spectrum for certain classes of self-similar
measures with overlaps has been computed explicitly by Ngai and Lau
\cite{MR1644468} and Ngai and Xie \cite{MR3897401}. The formulae
derived therein combined with \prettyref{cor:sefl-conformalWithOverlap}
will lead to a number of interesting formulae for the quantization
dimension.
\end{rem}

Let us now consider an \emph{inhomogeneous self-similar measure $\nu$}
given as the unique solution to 
\begin{equation}
\nu=p_{0}\mu+\sum_{i=1}^{N}p_{i}\nu\circ f_{i}^{-1}.\label{eq:InhomoSelfSimilarMeasure}
\end{equation}
In here, $\mu$ is a compactly supported Borel probability measure,
called the \emph{condensation measure}, $(f_{i})_{i=1}^{N}$ is a
family of contractive similitudes on $\mathbb{R}^{d}$ with contraction
ratios $\left(r_{i}\right)_{i=1,\ldots,N}$ and $\left(p_{i}\right)_{i=0,\ldots,N}\in(0,1)^{N}$
denotes a probability vector. We say the system fulfills the \emph{inhomogeneous
open set condition} (iOSC), if there exists a bounded non-empty open
set $U$ such that, $f_{i}\left(U\right)\subset U$, $\supp(\mu)\subset\overline{U}$,
$f_{i}\left(U\right)\cap\supp\mu=\emptyset$, $f_{i}\left(U\right)\cap f_{j}\left(U\right)=\emptyset$
for all $i\neq j$. The $L^{q}$-spectrum and multifractal properties
of such measures have first been treated in \cite{MR2440882,MR2285110}
and later in the generality needed for our purposes in \cite{MR3201323}.
This example has already been solved by Zhu in \cite{MR3345508} for
the special choice of a self-similar measure $\mu$, and he was also
able to bound the quantization coefficients away from zero and infinity
in this context. In the following we will need the function $\varrho_{\nu}:\left(0,1\right)\to\R$
given as the solution to $\ensuremath{\sum_{k=1}^{N}p_{k}^{q}r_{k}^{\varrho_{\nu}\left(q\right)}=1}$.
\begin{cor}
\label{cor:InhomIFS} Let $\nu$ be an inhomogeneous self-similar
measure with condensation measure $\mu$ under iOSC such that the
$L^{q}$-spectrum $\beta_{\mu}$ of $\mu$ exists as a limit on $(0,1)$.
Then the quantization dimension exists for every $r>0$ with
\[
D_{r}\left(\nu\right)=\frac{rq_{r}}{1-q_{r}},%\overline{\mathfrak{R}}_{\nu}\left(q_{r}\right).
\]
where  $q_{r}$ %and $\mathfrak{R}_{\nu}$ 
is defined in terms of the
$L^{q}$-spectrum of $\nu$, which is given, for $q\in\left(0,1\right)$,
by
\[
\beta_{\nu}\left(q\right)=\max\left\{ \beta_{\mu}\left(q\right),\varrho_{\nu}\left(q\right)\right\} .
\]
\end{cor}

In this example the $L^{q}$-spectrum is not necessarily differentiable
in $q_{r}$. Nevertheless, part \prettyref{enu:-LqRegularity1} of
the $L^{q}$-regularity for $r$ is still valid for $\beta_{\nu}$
under the assumptions of \prettyref{cor:InhomIFS}. Note that the
assumption on $\mu$ guarantees that the quantization dimension of
$\mu$ of order $r>0$ exists.
\begin{example}
\label{exa:ZhuiSelfSimilar} In our final example we consider an inhomogeneous
self-similar measure\emph{ $\nu$} as above, where instead of the
iOSC we assume that $\mu$ itself is a self-similar measure with respect
to the same IFS $\left(f_{i}\right)_{i=1}^{N}$ and probability vector
$\left(t_{i}\right)\in\left[0,1\right]^{N}$. We will demonstrate
how our result can be used to find the $L^{q}$-spectrum on $\left(0,1\right)$
knowing the quantization dimension. For fixed $r>0$ let $\varepsilon_{1,r}$
and $\varepsilon_{2,r}$ be implicitly given by 
\[
\sum_{i=1}^{N}\left(t_{i}s_{i}^{r}\right)^{\frac{\varepsilon_{1,r}}{\varepsilon_{1,r}+r}}=1\quad\text{and}\quad\sum_{i=1}^{N}\left(p_{i}s_{i}^{r}\right)^{\frac{\varepsilon_{2,r}}{\varepsilon_{2,r}+r}}=1.
\]
Then, by \cite[Theorem 1.2]{MR3454048}, we have $D_{r}(\text{\ensuremath{\nu)=\max\left\{ \varepsilon_{1,r},\varepsilon_{2,r}\right\} }}>0$.
Now, for $r=\beta_{\nu}(q)/q$ with $q\in(0,1)$, we have by \prettyref{thm:Main}
\[
q=q_{r}=\overline{D}_{r}\left(\nu\right)/\left(\overline{D}_{r}\left(\nu\right)+r\right).
\]
Without of loss of generality we assume $\max\left\{ \varepsilon_{1,r},\varepsilon_{2,r}\right\} =\varepsilon_{1,r}$.
Then, $q=\varepsilon_{1,r}/\left(r+\varepsilon_{1,r}\right)$ and

\[
1=\sum_{i=1}^{N}\left(t_{i}s_{i}^{r}\right)^{\frac{\varepsilon_{1,\beta_{\nu}(q)/q}}{\varepsilon_{1,\beta_{\nu}(q)/q}+r}}=\sum_{i=1}^{N}\left(t_{i}s_{i}^{\frac{\beta_{\nu}(q)}{q}}\right)^{q}=\sum_{i=1}^{N}t_{i}^{q}s_{i}^{\beta_{\nu}(q)}.
\]
Since $x\mapsto x/(x+r)$ is increasing on $\mathbb{R}_{\geq0}$,
it follows that $\beta_{\nu}(q)=\max\left\{ \rho_{1,q},\rho_{2,q}\right\} $
where $\rho_{1,q},\rho_{2,q}$ denote the unique solutions of

\[
\sum_{i=1}^{N}t_{i}^{q}s_{i}^{\rho_{1,q}}=1\quad\text{and}\quad\sum_{i=1}^{N}p_{i}^{q}s_{i}^{\rho_{2,q}}=1.
\]
This shows that also for this special case of inhomogeneous self-similar
measure the $L^{q}$-spectrum is given by the same formula as under
iOSC provided in \prettyref{cor:InhomIFS}.
\end{example}

\section{Partition functions , coarse multifractal formalism, and $L^{q}$-spectra
\label{sec:Partition-functions}}

In the sequel, we will refer to some results of \cite{KN2022} and
for this reason we recall the newly introduced notion of partition
functions with respect to a non-negative, monotone set function $\J$
defined on the dyadic cubes $\mathcal{D}$. The full generality of
this approach is not needed, and for the purposes of this paper we
will restrict ourselves throughout to the particular choice
\[
\J\coloneqq\J_{\nu,r}:Q\mapsto\nu\left(Q\right)\left(\Lambda\left(Q\right)\right)^{r/d},\;\text{for }r\geq0.
\]
The\emph{ partition function} with respect to $\J$ is then given,
for $q\geq0$, by 
\[
\GL_{\J}(q)\coloneqq\limsup_{n\rightarrow\infty}\GL_{\J,n}(q),\:\text{with}\:\GL_{\J,n}(q)\coloneqq\frac{\log\left(\sum_{C\in\mathcal{D}_{n}}\J(C)^{q}\right)}{\log(2^{n})},
\]
where we use the convention $0^{0}=0$, that is for $q=0$ we neglect
the summands with $\J\left(Q\right)=0$ in the definition of $\GL_{\J}$.
We define the critical value
\[
q_{\J}\coloneqq\inf\left\{ q\geq0:\GL_{\J}\left(q\right)<0\right\} 
\]
 which is connected to the $L^{q}$-spectrum and our notation from
\prettyref{sec:Introduction-and-background} via
\[
\GL_{\J_{\nu,r}}\left(q\right)=\beta_{\nu}\left(q\right)-rq\,\:\text{and }\;q_{r}=q_{\J_{\nu.r}}.
\]
In order to set up the\emph{ }optimize coarse multifractal dimension
with respect to $\J$ in this context, we introduce the following
terminology to simplify the notation and also to make the correspondence
with the relevant results in \autocite{KN2022b} explicit. For $n\in\N$
and $\alpha>0$, we set
\[
N_{\alpha,\J_{\nu,r}}\left(n\right)\coloneqq\left\{ C\in\mathcal{D}_{n}:\J_{\nu,r}(C)\geq2^{-\alpha n}\right\} =N_{\nu,\alpha,r}\left(n\right)
\]
and to express the dependence on $\mathfrak{J}_{\nu,r}$ for all relevant quantities we also use the notation $\mathcal{N}_{\alpha,\J_{\nu,r}}\left(n\right)\coloneqq\mathcal{N}_{\nu,\alpha,r}\left(n\right)$,
$\overline{F}_{\J_{\nu,r}}\left(\alpha\right)\coloneqq\overline{F}_{\nu,r}\left(\alpha\right)$,
$\underline{F}_{\J_{\nu,r}}\left(\alpha\right)\coloneqq\underline{F}_{\nu,r}\left(\alpha\right)$,
and write $\overline{F}_{\J_{\nu,r}}\coloneqq\overline{F}_{\nu,r}$
and $\overline{F}_{\J_{\nu,r}}\coloneqq\overline{F}_{\nu,r}$ to denote
the \emph{optimize coarse multifractal dimension} with respect to
$\J_{\nu,r}$. Similarly, for $x>1/\J_{\nu,r}(\Q)$, we write $\mathcal{M}_{\J_{\nu,r}}\left(x\right)\coloneqq\mathcal{M}_{\nu,r}\left(x\right)$
as given in \prettyref{eq:M-nu(r)} and for the corresponding \emph{partition
entropies} $\overline{h}_{\J_{\nu,r}}\coloneqq\overline{h}_{\nu,r}$
and $\underline{h}_{\J_{\nu,r}}\coloneqq\underline{h}_{\nu,r}$\emph{,
respectively.}

We end this section with some important facts about the \emph{$L^{q}$-spectrum
}of $\nu$:
\begin{itemize}
\item $\beta_{\nu}\left(1\right)=0$.
\item $\beta_{\nu}\left(0\right)=\overline{\dim}_{M}\left(\nu\right)$
\item $\dim_{\infty}(\nu)\leq d.$
\item For $q\geq0$, we have $-qd\leq\beta_{\nu}(q)$.
\item If $\nu$ is absolutely continuous, then $\beta_{\nu}\left(q\right)=d\left(1-q\right)$,
for all $q\in\left[0,1\right]$.
\end{itemize}

\section{Optimal partitions and partition entropy \label{sec:OptimalPartitions}}

\subsection{Upper bounds}

We make use of an observation from \autocite{KN2022b} which is valid
for set functions $\J:\mathcal{D}\to\R_{\geq0}$ on the dyadic cubes
$\mathcal{D}$, which are monotone, uniformly vanishing and locally
non-vanishing with $\J\left(\Q\right)>0$ and such that $\liminf_{n}\GL_{\J,n}\left(q\right)\in\R$
for some $q>0$. Here, \emph{uniformly vanishing} means $\lim_{k\to\text{\ensuremath{\infty}}}\sup_{Q\in\bigcup_{n\geq k}\mathcal{D}_{n}}\J\left(Q\right)=0$
and \emph{locally non-vanishing} means $\J\left(Q\right)>0$ implies
that there exists $Q'\in\left\{ R\in\mathcal{D}:R\subsetneq Q\right\} $
with $\J\left(Q'\right)>0$. It is important to note that all these
conditions on the set function are fulfilled for our particular choice
$\J=\J_{\nu,r}$.

Now we are in the position to state the upper bounds in terms of optimal
partitions denoted by $P_{t}$.
\begin{prop}[{\parencite[Prop. 4.1]{KN2022b}}]
\label{prop:GeneralResultOnPartitionFunction} For $0<t<\J(\Q)$,
we have that
\[
P_{t}\coloneqq\left\{ Q\in\mathcal{D}:\J\left(Q\right)<t\,\&\,\exists Q'\in\mathcal{D}_{\left|\log_{2}\Lambda\left(Q\right)\right|/d-1}^{N}:Q'\supset Q\,\&\J(Q')\geq t\right\} 
\]
is a finite partition of dyadic cubes of $\Q$, and we have
\[
\overline{F}_{\J}\leq\overline{h}_{\J}\leq\limsup_{t\downarrow0}\frac{\log\left(\card\left(P_{t}\right)\right)}{-\log(t)}\leq q_{\J},
\]
and 
\[
\underline{F}_{\J}\leq\underline{h}_{\J}\leq\liminf_{t\downarrow0}\frac{\log\left(\card\left(P_{t}\right)\right)}{-\log(t)}.
\]
\end{prop}

\subsection{The dual problem\label{subsec:DualProblem}}

In this section we recall the close connection to the shrinking rate
of the following quantity
\[
\gamma_{\J,n}\coloneqq\inf_{P\in\Pi_{\J},\card(P)\leq n}\max_{Q\in P}\J\left(Q\right)
\]
as elaborated in \cite{KN2022b}. As before, let $\J\coloneqq\J_{\nu,r}$,
$r>0$. Then \parencite[Cor. 4.10]{KN2022b} gives in this situation
the following.
\begin{lem}
\label{lem:BorzovBirmanSolomyak}For $r>0$, we have 
\[
\gamma_{\J_{\nu,r},n}=O\left(n^{-(1+r/d)}\right)\:\text{and}\:\mathcal{M}_{\J_{\nu,r}}(x)=O\left(x^{d/(d+r)}\right).
\]
If additionally $\nu$ is singular, then
\[
\gamma_{\J_{\nu,r},n}=o\left(n^{-(1+r/d)}\right)\:\text{and}\:\mathcal{M}_{\J_{\nu,r}}(x)=o\left(x^{d/(d+r)}\right).
\]
We obtain the following estimate for the upper exponent of divergence
of $\gamma_{\J,n}$ given by
\[
\overline{\alpha}_{\J_{\nu,r}}\coloneqq\limsup_{n\rightarrow\infty}\frac{\log\left(\gamma_{\J_{\nu,r},n}\right)}{\log(n)}\:\text{and}\:\underline{\alpha}_{\J_{\nu,r}}\coloneqq\liminf_{n\rightarrow\infty}\frac{\log\left(\gamma_{\J_{\nu,r},n}\right)}{\log(n)}.
\]
\end{lem}

\begin{prop}[{\parencite[Prop. 4.11]{KN2022b}}]
\label{prop:UpperBoundJ_a} For $r>0$, we have
\[
-\frac{1}{\overline{h}_{\J_{\nu,r}}}=\overline{\alpha}_{\J_{\nu,r}}\leq-\frac{1}{q_{\J_{\nu,r}}}\leq-\frac{\overline{\dim}_{M}(\nu)+r}{\overline{\dim}_{M}(\nu)}\leq-(1+r/d)\quad\text{ and }\quad-\frac{1}{\underline{h}_{\J_{\nu,r}}}=\underline{\alpha}_{\J_{\nu,r}}.
\]
\end{prop}

\subsection{Lower bounds\label{sec:Coarse-multifractal-analysis}}

We have seen so far that 
\[
\overline{F}_{r}=\overline{F}_{\J_{\nu,r}}\leq q_{\J_{\nu,r}}=q_{r}.
\]
For the lower bounds we use ideas from the coarse multifractal analysis
as developed in \autocite{KN2022,KN2022b} for the particular choice
$\J=\J_{\nu,r}$. Since the conditions stated \autocite[Prop.  5.7 and 5.8]{KN2022b}
are fulfilled for $\J_{\nu,r}$, we obtain the following two equalities.
\begin{prop}
\label{prop:LowerBoundUpper/lowerSpecDim} For $r>0$, we always have
$\overline{F}_{r}=q_{r}$ and if $\nu$ is \emph{$L^{q}$-regular
for $r$, then} $\underline{F}_{r}=q_{r}.$
\end{prop}

\section{\label{sec:Upper-bounds-1} Application to the Quantization Dimension}

In this section we give a proof for the upper and lower bound of the
quantization dimension of order $r>0$ as stated in \prettyref{thm:Main}.

\subsection{Upper bounds\label{subsec:Upper-bound-for-spectral-dimension-1}}

We start by proving the upper bound for the upper quantization dimension.
\begin{prop}
\label{prop:UpperBound}For all $n\in\N$, we have 
\[
\er_{n,r}(\nu)^{r}\leq\sqrt{d}n\gamma_{\J_{\nu,r},n}.
\]
In particular, 
\[
\overline{D}_{r}\left(\nu\right)\leq\frac{r\overline{h}_{\J_{\nu,r}}}{1-\overline{h}_{\J_{\nu,r}}}=\frac{rq_{r}}{1-q_{r}}=\overline{\mathfrak{R}}_{\nu}\left(q_{r}\right)\le\overline{\dim}_{M}(\nu),
\]
and 
\[
\underline{D}_{r}\left(\nu\right)\leq\frac{r\underline{h}_{\J_{\nu,r}}}{1-\underline{h}_{\J_{\nu,r}}}.
\]
\end{prop}

\begin{proof}
We only consider the case $q_{r}>0$. The case $q_{r}=0$ follows
analogously. Let $P\in\Pi_{\nu}$ with $\card(P)\leq n$. Write $P=\left\{ Q_{1},\ldots,Q_{\card(P)}\right\} $
and let $m_{i}$ denote the middle point of the dyadic cube $Q_{i}$,
$i\leq\card(P)$ and set $\alpha_{n}\coloneqq\left(m_{1},\ldots,m_{\card(P)}\right)$.
Then we have 
\begin{align*}
\er_{n,r}(\nu) & \leq\left(\int d(x,\alpha_{n})^{r}\d\nu(x)\right)^{1/r}=\left(\sum_{i=1}^{\card(P)}\int_{Q_{i}}d(x,\alpha_{n})^{r}\d\nu(x)\right)^{1/r}\\
 & \leq\sqrt{d}\left(\sum_{i=1}^{\card(P)}\nu(Q_{i})\Lambda(Q_{i})^{r/d}\right)^{1/r}\leq\sqrt{d}n^{1/r}\left(\max_{Q\in P}\nu(Q)\Lambda(Q)^{r/d}\right)^{1/r}.
\end{align*}
 Now, taking the infimum over all $P\in\Pi_{\nu}$ with $\card(P)\leq n$
delivers 
\[
\er_{n,r}^{r}(\nu)\leq\sqrt{d}n\gamma_{\J_{\nu,r},n}.
\]
Note that by \prettyref{prop:UpperBoundJ_a}, for every $\varepsilon>0$,
we have for $n$ large 
\[
n\gamma_{\J_{\nu,r},n}\leq n^{1-1/q_{r}+\varepsilon},
\]
and there exists a subsequence $(n_{k})$ such that
\[
n_{k}\gamma_{\J_{\nu,r},n_{k}}\leq n_{k}^{1-1/\underline{h}_{\J_{\nu,r}}+\varepsilon},
\]
implying
\[
\limsup_{n\rightarrow\infty}\frac{-\log(n)}{\log\left(\er_{n,r}(\nu)\right)}\leq\frac{r\overline{h}_{\J_{\nu,r}}}{1-\overline{h}_{\J_{\nu,r}}}=\frac{rq_{r}}{1-q_{r}}\:\text{and}\:\liminf_{k\rightarrow\infty}\frac{-\log(n_{k})}{\log\left(\er_{n_{k},r}(\nu)\right)}\leq\frac{r\underline{h}_{\J_{\nu,r}}}{1-\underline{h}_{\J_{\nu,r}}}.
\]
Moreover, \prettyref{prop:UpperBoundJ_a} implies $-1/q_{r}\leq-\left(\overline{\dim}_{M}(\nu)+r\right)/\overline{\dim}_{M}(\nu)$
which proves the last inequality.
\end{proof}
\begin{cor}
If $\nu$ is singular, then $\lim_{n\rightarrow\infty}n^{1/d}\er_{n,r}(\nu)=0.$
\end{cor}

\begin{proof}
Since $\nu$ is singular, by \prettyref{lem:BorzovBirmanSolomyak}
and \prettyref{prop:UpperBound}, we have
\[
\er_{n,r}(\nu)\leq\left(n\gamma_{\J_{\nu,r},n}\right)^{1/r}=o\left(n^{-1/d}\right).
\]
\end{proof}
Using the simple fact $\overline{D}_{0}\left(\nu\right)\leq\overline{D}_{r}\left(\nu\right)$,
$\underline{D}_{\text{0}}\left(\nu\right)\leq\underline{D}_{r}\left(\nu\right)$
(see also \cite[Lemma 3.5]{MR2055056}) and \prettyref{prop:UpperBound},
we obtain the following bounds for the geometric mean error.
\begin{cor}
We have $\overline{D}_{0}\left(\nu\right)\leq-\partial^{-}\beta_{\nu}(1)$
and $\underline{D}_{0}\left(\nu\right)\leq\liminf_{r\downarrow0}r\underline{h}_{\J_{\nu,r}}/\left(1-\underline{h}_{\J_{\nu,r}}\right).$
\end{cor}

\subsection{Lower bounds\label{subsec:Lower-bound-for-spectral-dimension}}

Recall, for $s>0$ we let $\left\langle Q\right\rangle _{s}$ denote
the cube centered and parallel with respect to $Q$ such that $\Lambda(Q)=s^{-d}\Lambda\left(\left\langle Q\right\rangle _{s}\right)$,
$s>0$. The following lemma was implicitly used in \cite{KN2022}.
\begin{lem}
\label{lem:GeneralPrincipleLowerBound-1} For fixed $\alpha>0$ there
exists a sequence $\left(E_{\alpha,n}\right)$ with $E_{\alpha,n}\subset N_{\alpha,\J}\left(n\right)$,
$c_{\alpha,n}\coloneqq\card\left(E_{\alpha,n}\right)\text{\ensuremath{\geq\left\lfloor \mathcal{N}_{\alpha,\J}\left(n\right)/5^{d}\right\rfloor }}$
and for all cubes $Q,Q'\in E_{\alpha,n}$ with $Q\neq Q'$ we have
$\left\langle \mathring{Q}\right\rangle _{3}\cap\left\langle \mathring{Q'}\right\rangle _{3}=\emptyset$.
\end{lem}

\begin{proof}
If $N_{\alpha,\J}\left(n\right)=\emptyset$, nothing needs to be shown.
Hence, we assume $N_{\alpha,\J}\left(n\right)\neq\emptyset$ and construct
inductively a subset $E_{n}$ of $N_{\alpha,\J}\left(n\right)$ of
cardinality $c_{n}\coloneqq\card\left(E_{\alpha,n}\right)\text{\ensuremath{\geq\left\lfloor \mathcal{N}_{\alpha,\J}\left(n\right)/5^{d}\right\rfloor }}$
such that for all cubes $Q,Q'\in E_{n}$ with $Q\neq Q'$ we have
$\left\langle \mathring{Q}\right\rangle _{3}\cap\left\langle \mathring{Q'}\right\rangle _{3}=\emptyset$.
At the beginning of the induction we set $D^{\left(0\right)}\coloneqq N_{\alpha,\J}\left(n\right)$.
Assume we have constructed $D^{\left(0\right)}\supset D^{\left(1\right)}\supset\cdots\supset D^{\left(j-1\right)}$
such that \textbf{$\left\langle \mathring{Q_{j}}\right\rangle _{5}\cap\mathring{Q}\neq\emptyset$},\textbf{
}for some $Q,Q_{j}\in D^{\left(j-1\right)}$ with $Q\neq Q_{j}$.
In this case we set
\[
D^{(j)}\coloneqq\left\{ C\in D^{\left(j-1\right)}:\mathring{C}\cap\left\langle \mathring{Q_{j}}\right\rangle _{5}=\emptyset\right\} \cup\left\{ Q_{j}\right\} .
\]
By this construction, we have $\card\left(D^{(j)}\right)<\card\left(D^{(j-1)}\right)$,
since $\mathring{Q}\cap\left\langle \mathring{Q_{j}}\right\rangle _{5}\neq\emptyset$.
Otherwise, if \textbf{$\left\langle \mathring{Q}\right\rangle _{5}\cap\mathring{Q}'=\emptyset$
}for all $Q,Q'\in D^{\left(j-1\right)}$ with $Q\neq Q'$, then the
finite induction terminates and we set $E_{\alpha,n}=D^{\left(j-1\right)}$.
In each inductive step, we remove at most $5^{d}-1$ elements of $D^{\left(j-1\right)}$,
while one element, namely $Q_{j}$, is kept. This implies $\card\left(E_{\alpha,n}\right)\text{\ensuremath{\geq\left\lfloor \mathcal{N}_{\alpha,\J}\left(n\right)/5^{d}\right\rfloor }}.$
\end{proof}
\begin{prop}
\label{prop:QDLowerBound}For $r>0$, we have
\[
\overline{D}_{r}\left(\nu\right)\geq\frac{rq_{r}}{1-q_{r}}\quad\text{and }\quad\underline{D}_{r}\left(\nu\right)\geq\frac{r\underline{F}_{\J_{\nu,r}}}{1-\underline{F}_{\J_{\nu,r}}}.
\]
\end{prop}

\begin{proof}
Fix $\alpha>0$ such that $\overline{F}_{\J_{\nu,r}}\left(\alpha\right)>0$.
Further, let $\left(n_{k}\right)_{k}$ be such that 
\[
\overline{F}_{\J_{\nu,r}}\left(\alpha\right)=\lim_{k\rightarrow\infty}\frac{\log^{+}\left(\mathcal{N}_{\alpha,\J_{\nu,r}}\left(n_{k}\right)\right)}{\log2^{n_{k}}}
\]
and let $c_{\alpha,n_{k}}\coloneqq\card E_{\alpha,n_{k}}$ be given
as in \prettyref{lem:GeneralPrincipleLowerBound-1} for $\J=\J_{\nu,r}$.
Notice that by our assumption $\overline{F}_{\J_{\nu,r}}\left(\alpha\right)>0$
and it follows that $\lim_{k}c_{\alpha,n_{k}}=\infty$. Let $A$ be
of cardinality at most $c_{\alpha,n_{k}}/2$ and 
\[
E_{\alpha,n_{k}}'\coloneqq\left\{ Q\in E_{\alpha,n_{k}}:\inf_{a\in A}d\left(a,Q\right)\geq2^{-n_{k}}\right\} .
\]
Since, for all $Q_{1},Q_{2}\in E_{\alpha,n_{k}}$ we have $\left\langle \mathring{Q}_{1}\right\rangle _{3}\cap\left\langle \mathring{Q}_{2}\right\rangle _{3}=\emptyset$
it follows that if $d\left(a,Q\right)<2^{-n_{k}}$ for some $a\in A$
and $Q\in E_{\alpha,n_{k}}$, then $d\left(a,Q'\right)\geq2^{-n_{k}}$
for all $Q'\in E_{\alpha,n_{k}}\setminus\left\{ Q\right\} $ and therefore,
\[
\card\left\{ Q\in E_{\alpha,n_{k}}:\inf_{a\in A}d\left(a,Q\right)<2^{-n_{k}}\right\} \leq\card\left(A\right).
\]
 Hence, $\card\left(E_{\alpha,n_{k}}'\right)\geq c_{\alpha,n_{k}}/2$
and
\[
\int d(x,A)^{r}\d\nu(x)\geq\sum_{Q\in E_{\alpha,n_{k}}'}\int_{Q}d(x,A)^{r}\d\nu(x)\geq\sum_{Q\in E_{\alpha,n_{k}}'}\nu\left(Q\right)\Lambda\left(Q\right)^{r/d}\geq c_{\alpha,n_{k}}2^{-\alpha n_{k}-1}.
\]
Consequently, $\er_{\left\lfloor c_{\alpha,n_{k}}/2\right\rfloor ,r}^{r}(\nu)\geq c_{\alpha,n_{k}}2^{-\alpha n_{k}-1}$
and we obtain for the first claim
\begin{align*}
\limsup_{k}\frac{\log\left\lfloor c_{\alpha,n_{k}}/2\right\rfloor }{-\log\er_{\left\lfloor c_{\alpha,n_{k}}/2\right\rfloor ,r}(\nu)} & \geq\limsup_{k}\frac{r\log\left(c_{\alpha,n_{k}}/4\right)}{-\log\er_{\left\lfloor c_{\alpha,n_{k}}/2\right\rfloor ,r}(\nu)}\\
 & \geq\limsup_{k}\frac{r\left(\log\left(\mathcal{N}_{\alpha,\J_{\nu,r}}\left(n_{k}\right)/5^{d}-1\right)-\log4\right)}{-\log\left(\mathcal{N}_{\alpha,\J_{\nu,r}}\left(n_{k}\right)/5^{d}-1\right)+\left(\alpha n_{k}+1\right)\log2}\\
 & =\frac{r\overline{F}_{\J_{\nu,r}}\left(\alpha\right)/\alpha}{1-\overline{F}_{\J_{\nu,r}}\left(\alpha\right)/\alpha}.
\end{align*}
This gives 
\[
\limsup_{n}\frac{\log n}{-\log\er_{n,r}(\nu)}\geq\sup_{\alpha>0}\limsup_{k}\frac{\log\left\lfloor c_{\alpha,n_{k}}/2\right\rfloor }{-\log\er_{\left\lfloor c_{\alpha,n_{k}}/2\right\rfloor ,r}(\nu)}\geq\sup_{\alpha>0}\frac{r\overline{F}_{\J_{\nu,r}}\left(\alpha\right)/\alpha}{1-\overline{F}_{\J_{\nu,r}}\left(\alpha\right)/\alpha}=\frac{r\overline{F}_{\J_{\nu,r}}}{1-\overline{F}_{\J_{\nu,r}}}.
\]
For the lower limit assume $\underline{F}_{\J_{\nu,r}}\left(\alpha\right)>0$
and note that for every $\epsilon\in\left(0,\underline{F}_{\J_{\nu,r}}\left(\alpha\right)\right)$
and all $n$ large 
\[
c_{\alpha,n}\coloneqq\left\lceil 5^{-d}\mathcal{N}_{\alpha,\J_{\nu,r}}\left(n\right)\right\rceil \geq2^{n\left(\underline{F}_{\J_{\nu,r}}\left(\alpha\right)-\epsilon\right)}.
\]
 Now, for $k\in\N$, we define
\[
n_{k}\coloneqq\left\lceil \frac{\log\left(2k\right)}{\left(\underline{F}_{\J_{\nu,r}}\left(\alpha\right)-\epsilon\right)\log(2)}\right\rceil .
\]
 Clearly, this gives $c_{\alpha,n_{k}}\geq2k$. Then for any subset
$A$ with $\card\left(A\right)\leq k\leq c_{\alpha,n_{k}}/2$ we have
as above $\card\left(E_{\alpha,n_{k}}'\right)\geq c_{\alpha,n_{k}}/2\geq k$.\textbf{
}Then
\[
\int d(x,A)^{r}\d\nu(x)\geq\sum_{Q\in E_{\alpha,n_{k}}'}\int_{Q}d(x,A)^{r}\d\nu(x)\geq\sum_{Q\in E_{\alpha,n_{k}}'}\nu\left(Q\right)\Lambda\left(Q\right)^{r/d}\geq c_{\alpha,n_{k}}2^{1-\alpha n_{k}}\geq k2^{-\alpha n_{k}}.
\]
Taking the infimum over $A$ with $\card\left(A\right)\leq k$ we
obtain $\er_{k,r}^{r}\left(\nu\right)\ensuremath{\geq}k2^{-\alpha n_{k}}$.
This gives
\[
\frac{\log k}{-\log\er_{k,r}\left(\nu\right)}\geq\frac{r\log k}{-\log k+\alpha n_{k}\log\left(2\right)}\geq\frac{r\log k}{-\log k+\alpha\log\left(2k\right)/\left(\underline{F}_{\J_{\nu,r}}\left(\alpha\right)-\epsilon\right)}.
\]
Taking the lower limit over $k$ and letting $\epsilon$ tend to zero,
yields
\[
\underline{D}_{r}\left(\nu\right)\geq\frac{r}{-1+\alpha/\underline{F}_{\J_{\nu,r}}\left(\alpha\right)}=\frac{r\underline{F}_{\J_{\nu,r}}\left(\alpha\right)/\alpha}{1-\underline{F}_{\J_{\nu,r}}\left(\alpha\right)/\alpha}.
\]
Finally, taking the supremum for $\alpha>0$ gives

\[
\underline{D}_{r}\left(\nu\right)\geq\frac{r\underline{F}_{\J_{\nu,r}}}{1-\underline{F}_{\J_{\nu,r}}}.
\]
\end{proof}

\subsection{Proofs of main results\label{subsec:Proofs-of-main}}
\begin{proof}
[Proof of  \prettyref{thm:Main}] The main theorem follows by combining
\prettyref{prop:UpperBound} and \prettyref{prop:QDLowerBound}.
\end{proof}
\begin{proof}
[Proof of  \prettyref{cor:DrContinuous}] This follows from the
fact that $r\mapsto q_{r}$ is continuous. Indeed, if $q_{r}=0$ for
some $r>0$, then $0\leq\beta_{\nu}(q)<rq$ for all $q\in(0,1).$
Consequently, $\lim_{q\downarrow0}\beta_{\nu}(q)=0$ and combined
with the convexity of $\beta_{\nu}$ and $\beta_{\nu}(1)=0$, we infer
$\beta_{\nu}\left(q\right)=0$ for $q>0$. Therefore, $q_{r}=0$ for
all $r>0$. The case $q_{r}>0$ follows from the fact that $\beta_{\nu}$
is continuous and decreasing on $\left(0,1\right]$ with $\beta_{\nu}(1)=0$.
\end{proof}
\begin{proof}
[Proof of  \prettyref{cor:rigidity}] This follows from the convexity
of $\beta_{\nu}$ combined with \prettyref{thm:Main}.
\end{proof}
\begin{proof}
[Proof of  \prettyref{thm:LqRegularImpliesRegular}] By \prettyref{prop:LowerBoundUpper/lowerSpecDim},
we have $\underline{F}_{\J_{\nu,r}}=q_{r}$. Hence, we can infer from
\prettyref{thm:Main} 
\[
\frac{rq_{r}}{1-q_{r}}=\frac{r\underline{F}_{\J_{\nu,r}}}{1-\underline{F}_{\J_{\nu,r}}}\leq\underline{D}_{r}(\nu)\leq\overline{D}_{r}(\nu)\leq\frac{rq_{r}}{1-q_{r}}.
\]
\end{proof}
\begin{proof}
[Proof of  \prettyref{cor:mixed measure}] This follows from the
convexity of $\beta_{\nu}$ combined with \prettyref{thm:Main}.
\end{proof}
\begin{proof}
[Proof of  \prettyref{cor:sefl-conformalWithOverlap}] The corollary
follows by combining \prettyref{thm:Main} with \prettyref{thm:LqRegularImpliesRegular}
and the regularity result on the $L^{q}$-spectrum obtained in \cite{MR2322179}.
\end{proof}
\begin{proof}
[Proof of  \prettyref{prop:upperBoundForLowerDr}] We only consider
the case $\dim_{\infty}(\nu)>0$. The case $\dim_{\infty}(\nu)=0$
follows along the same lines. Let $0<s<\dim_{\infty}(\nu)$. Then,
for $n$ large, we have 
\[
\max_{Q\in\mathcal{D}_{n}}\nu(Q)\Lambda(Q)^{r/d}\leq2^{-(s+r)n}<2^{-(s+r)n+1}.
\]
This implies $\mathcal{M}_{\J_{\nu,r}}\left(2^{-(s+r)n+1}\right)\leq2^{n\beta_{\nu,n}(0)}.$
Therefore, we obtain
\begin{align*}
\underline{h}_{\J_{\nu,r}} & \leq\liminf_{n\rightarrow\infty}\frac{\log\left(\mathcal{M}_{\J_{\nu,r}}\left(2^{(s+r)n-1}\right)\right)}{\log\left(2^{(s+r)n-1}\right)}\leq\liminf_{n\rightarrow\infty}\frac{\log\left(2^{n\beta_{\nu,n}(0)}\right)}{\log\left(2^{(s+r)n-1}\right)}\\
 & =\liminf_{n\rightarrow\infty}\frac{\beta_{\nu,n}(0)}{(s+r)-1/n}=\frac{\underline{\dim}_{M}(\nu)}{r+s}.
\end{align*}
Now, $s\uparrow\dim_{\infty}(\nu)$ proves the claim.
\end{proof}
\begin{proof}
[Proof of  \prettyref{cor:InhomIFS}] This follows immediately from
\cite[Corollary 3.10]{MR3201323}, where the existence of $\beta_{\nu}$
as a limit and $\beta_{\nu}=\max\left\{ \beta_{\mu},\varrho_{\nu}\right\} $
is shown, combined with \prettyref{thm:Main} and \prettyref{thm:LqRegularImpliesRegular}.
\end{proof}
\printbibliography

\end{document}